\documentclass[12pt,oneside,reqno]{amsart}
\usepackage{amsmath,amssymb,amsthm}
\usepackage[margin=1.1in]{geometry}
\usepackage{booktabs}
\usepackage{xurl}
\usepackage[hidelinks]{hyperref}
\hypersetup{pdftitle={Fuglede's conjecture holds in eight dimensions over the field with two elements},pdfauthor={Samuel J. Ferguson and Joshua Iosevich}}
\numberwithin{equation}{section}
\newtheorem{theorem}{Theorem}[section]
\newtheorem{lemma}[theorem]{Lemma}
\newtheorem{proposition}[theorem]{Proposition}
\newtheorem{corollary}[theorem]{Corollary}
\theoremstyle{definition}
\newtheorem{definition}[theorem]{Definition}
\newtheorem{question}[theorem]{Question}

\theoremstyle{remark}

\title[Fuglede's conjecture holds in eight dimensions]{Fuglede's conjecture holds in eight dimensions over the field with two elements}
\author{Samuel J. Ferguson}
\address{Department of Mathematics, New York University}
\email{sjf370@nyu.edu}
\author{Joshua Iosevich}
\address{School of Mathematics and Statistics, Rochester Institute of Technology, 1 Lomb Memorial Drive, Rochester, NY 14623, USA}
\email{ji1964@rit.edu}
\date{October 26, 2026}

\begin{document}
\begin{abstract}
Fuglede's conjecture in a finite abelian group says that a nonempty subset tiles by translation if and only if it is spectral, meaning that the space of complex-valued functions on the subset has an orthogonal basis consisting of restrictions of group characters. We prove the conjecture in the eight-dimensional vector space over the field with two elements. In fact, every spectral set and every tile in this space has a tiling complement which is a linear subspace. We first prove that the Fourier support of a nonempty proper subset of density different from one half contains a two-dimensional subspace. Consequently, every spectral set of density one quarter has a subspace tiling complement. This gives a simpler proof of Kadowaki's dimension-seven result, replacing his use of the classification of Hadamard matrices of order 32 by the classical five classes of order 16. Our dimension-eight proof uses exact finite verifications involving Hadamard matrices of order 32 and perfect codes of length 15. We include the proofs of the reductions used from Kadowaki's preprint, so that neither positive result depends on his dimension-seven theorem. Finally, we give an elementary spectral example of density one eighth in dimension nine whose Fourier support contains no three-dimensional subspace. Thus, the Fourier-support argument at quarter density does not extend in this form to eighth density. Dimension nine is the only case still unresolved over the field with two elements, and this obstruction shows why the present Fourier-support method cannot resolve it as it stands.
\end{abstract}
\maketitle
{\small\setlength{\parskip}{0pt}\tableofcontents}

\section{Introduction}
\label{sec:introduction}

In this paper, we prove Fuglede's conjecture in the eight-dimensional vector space over the field with two elements. We also give a new, simpler proof of the seven-dimensional case, recently established in a preprint of Kadowaki \cite{Kadowaki2026}. Our main observation concerns spectral sets of density one quarter and applies in every dimension. We show, however, that the corresponding observation at density one eighth is false, even for a spectral set which tiles.

We briefly review Fuglede's conjecture \cite{Fuglede1974} in its original setting. Let $E$ be a measurable subset of $\mathbb R^d$ of finite, positive Lebesgue measure. We call $E$ a \emph{spectral set} if $L^2(E)$ has an orthogonal basis of complex exponentials
\[
 x\longmapsto e^{2\pi i\lambda\cdot x},\qquad \lambda\in\Lambda,
\]
for some set $\Lambda\subset\mathbb R^d$. The set $\Lambda$ is called a \emph{spectrum} for $E$. On the other hand, $E$ \emph{tiles by translation} if there is a set $T$ such that the translates $E+t$, $t\in T$, partition $\mathbb R^d$ up to sets of measure zero. Fuglede conjectured that these two properties are equivalent. The conjecture arose from his study of commuting self-adjoint realizations of partial differential operators: spectral sets provide a description of the domains on which such realizations exist, and tiling would provide a geometric description of the same domains.

Tao \cite{Tao2004} disproved the conjecture in dimensions five and higher using constructions in finite groups. Subsequent work gave counterexamples in lower dimensions and in the other direction; see \cite{Matolcsi2005,KolountzakisMatolcsi2006,FarkasMatolcsiMora2006}. These constructions also explain the interest in the finite-group version of the problem. In that setting, both spectrality and tiling can be expressed by finite systems of equations, but the relation between them remains difficult.

We write $\mathbb Z_p$ for the field of integers modulo a prime $p$. The vector space $\mathbb Z_p^d$ is considered as an additive group. If $E$ is a nonempty subset of this group, we say that $E$ \emph{tiles} if there is a subset $T$ for which every group element has exactly one representation $e+t$, with $e\in E$ and $t\in T$. We write
\[
 E\oplus T=\mathbb Z_p^d
\]
and call $T$ a \emph{tiling complement} of $E$. In particular, $|E||T|=p^d$. Here and below, $|A|$ denotes the cardinality of a finite set $A$.

The characters of $\mathbb Z_p^d$ are the functions $x\mapsto e^{2\pi i\lambda\cdot x/p}$, where $\lambda\in\mathbb Z_p^d$ and $\lambda\cdot x=\lambda_1x_1+\cdots+\lambda_dx_d$. The space $L^2(E)$ is the $|E|$-dimensional space of complex-valued functions on $E$, with inner product
\[
 \langle f,g\rangle=\sum_{x\in E}f(x)\overline{g(x)}.
\]
We call $E$ \emph{spectral} if the restrictions of the characters with frequencies in some set $\Lambda$ form an orthogonal basis of $L^2(E)$. Necessarily $|\Lambda|=|E|$. We again call $\Lambda$ a spectrum and call $(E,\Lambda)$ a \emph{spectral pair}. Fuglede's conjecture in $\mathbb Z_p^d$ asserts that its nonempty spectral sets are precisely its tiles.

Iosevich, Mayeli, and Pakianathan \cite{IosevichMayeliPakianathan2017} proved the conjecture in dimension two for every prime; another approach appears in \cite{KissEtAl2020}. Aten et al.\ \cite{AtenEtAl2017} developed the theory over prime fields, including the relation with log-Hadamard matrices, and obtained counterexamples in dimension five for every odd prime and in dimension four for primes congruent to three modulo four. Ferguson and Sothanaphan \cite{FergusonSothanaphan2020} extended the latter construction to all odd primes. They also proved the conjecture over $\mathbb Z_2$ in dimension four, reported computational proofs in dimensions five and six, and modified Tao's counterexample to disprove it in dimension ten. Further details of their computations were given in \cite{FergusonSothanaphan2019Addendum}. Their work left dimensions seven, eight, and nine unresolved over $\mathbb Z_2$.

Kadowaki's recent preprint \cite{Kadowaki2026} resolves dimension seven. His proof introduces three useful ingredients: a common spectrum for all tiling complements of a set with at most eight points, a restriction on spectral cardinalities obtained from integer matrix operations, and a projection from a spectral set onto a point set attached to a dephased log-Hadamard matrix. His proof of the seven-dimensional case also uses a scan of the complete classification of Hadamard matrices of order 32. We retain his three general ingredients, with proofs, and replace the order-32 calculation in dimension seven by our quarter-density argument. Thus we give a new proof of his theorem, while attributing the original result to him. His preprint has not yet been refereed; the proofs below do not assume its main theorem.

Our principal result is the following.

\begin{theorem}\label{thm:main}
Let $d=7$ or $d=8$, and let $E$ be a nonempty subset of $\mathbb Z_2^d$. The following are equivalent:
\begin{enumerate}
\item $E$ is spectral;
\item $E$ tiles $\mathbb Z_2^d$;
\item $E$ has a tiling complement which is a linear subspace.
\end{enumerate}
\end{theorem}

For the remainder of the paper we work over $\mathbb Z_2$, except when discussing earlier results. Put $G=\mathbb Z_2^d$ and $N=2^d$. Vector addition and dot products are taken modulo two. For a subset $A$ of $G$, put
\[
 \begin{gathered}
 A+A=\{a+a':a,a'\in A\},\qquad W_A(u)=\sum_{a\in A}(-1)^{u\cdot a},\\
 \operatorname{supp}W_A=\{u\in G:W_A(u)\ne0\}.
 \end{gathered}
\]
Here $A+A$ is the \emph{sumset} of $A$ with itself, which is also its difference set over $\mathbb Z_2$. Thus $W_A$ is the unnormalized Fourier transform of the indicator of $A$, also called the \emph{Walsh transform} or \emph{Walsh--Hadamard transform}; the characters used here are the finite form of Walsh's orthogonal system~\cite{Walsh1923}. All sums and inner products use counting measure. We call $|A|/N$ the \emph{density} of $A$. The zero vector belongs to $\operatorname{supp}W_A$ whenever $A$ is nonempty. By a plane we mean a two-dimensional linear subspace; an affine plane will always be called an affine plane.

The two-dimensional subspace lemma, Lemma~\ref{lem:support-plane}, states that $\operatorname{supp}W_A$ contains a plane whenever $A$ is nonempty and proper and $|A|\ne N/2$. Its proof uses only the identity $1_A^2=1_A$. Applied to a spectrum of a set $E$ with $|E|=N/4$, it gives the quarter-density theorem, Theorem~\ref{thm:quarter}: there is a plane $P$ such that $E\oplus P=G$. At greater density, the same argument leaves only the familiar half-density and whole-group cases. In dimension seven this settles the spectral cardinalities 32 and 64 without a classification of the corresponding Hadamard matrices.

The dimension-eight proof still has two larger finite parts. The spectral sets with 32 points are treated using the classification of Hadamard matrices of order 32 by Kharaghani and Tayfeh-Rezaie \cite{KharaghaniTayfehRezaie2013}, in the form distributed by McKay \cite{McKayData}. The tiles with 16 points are treated using the classification of perfect codes of length 15 by \"Osterg\aa rd and Pottonen \cite{OstergardPottonen2009}. We explain why these classifications cover every set under consideration and give exact certificates for the resulting tilings. A check of a certificate and a check that all required cases have been included are separate matters; both are addressed in the appendix and accompanying programs. The completeness of the published classifications themselves is an external input.

The dimension-eight result and the quarter-density argument were first included in our broader preliminary preprint \cite{FergusonIosevich2026}. The present paper gives a standalone treatment of these results, a simpler proof in dimension seven, and an elementary form of the obstruction at eighth density. We do not include the odd-prime argument or the further dimension-nine reductions from that preliminary preprint. In particular, this paper does not assert that Fuglede's conjecture is resolved in dimension nine.

We begin in Section~\ref{sec:quarter} with the Fourier identities and the quarter-density result. Section~\ref{sec:small-sets} treats projections and small tiling partners. Section~\ref{sec:hadamard} gives the cardinality bound and the reduction to canonical point sets. Section~\ref{sec:seven} then proves the dimension-seven case. Section~\ref{sec:eight} treats the two additional finite problems needed in dimension eight and proves Theorem~\ref{thm:main}. In Section~\ref{sec:eighth} we give the eighth-density example. We conclude with further questions. The appendix records the small matrices, the finite verifications, and instructions for reproducing them.

\section{Fourier transforms and spectral sets of large size}
\label{sec:quarter}

We first give the Fourier observations used in the proofs below. Let
$G=\mathbb{Z}_2^d$ and $N=2^d$. For a function $f:G\to\mathbb{C}$,
we use the unnormalized Fourier transform, or Walsh transform,
\[
 W_f(u)=\sum_{x\in G} f(x)(-1)^{u\cdot x},\qquad u\in G.
\]
When $f=1_A$ is the indicator of a subset $A$ of $G$, we write $W_A$
instead of $W_{1_A}$. The \emph{support} of $W_A$ is the set
\[
 \operatorname{supp}W_A=\{u\in G:W_A(u)\ne0\}.
\]
In particular, $0\in\operatorname{supp}W_A$ whenever $A$ is nonempty.
Notice that $x-y=x+y$ in $G$. Thus, throughout this paper, the difference
set of $A$ is written as $A+A$.

The character orthogonality identity is
\begin{equation}
 \sum_{u\in G}(-1)^{u\cdot x}=
 \begin{cases}
 N,&x=0,\\
 0,&x\ne0.
 \end{cases}
 \label{eq:character-orthogonality}
\end{equation}
Indeed, if $x\ne0$, choose a coordinate in which $x$ is one and pair each
$u$ with the vector obtained by changing that coordinate. The two terms
cancel. It follows that Fourier inversion has the form
\begin{equation}
 f(x)=\frac1N\sum_{u\in G}W_f(u)(-1)^{u\cdot x}.
 \label{eq:fourier-inversion}
\end{equation}

Suppose that $(E,\Lambda)$ is a spectral pair, with $|E|=|\Lambda|=m$.
Its character matrix $H=((-1)^{\lambda\cdot e})_{\lambda\in\Lambda,e\in E}$
satisfies $HH^{\mathsf T}=mI$. Since $H$ is square, it is invertible,
and $H^{\mathsf T}H=mI$ as well. Hence $(\Lambda,E)$ is also a spectral
pair. Equivalently,
\begin{equation}
 \begin{split}
 (E+E)\setminus\{0\}&\subseteq\{u:W_\Lambda(u)=0\},\\
 (\Lambda+\Lambda)\setminus\{0\}&\subseteq\{u:W_E(u)=0\}.
 \end{split}
 \label{eq:spectral-zero-sets}
\end{equation}
We will use both parts of this symmetry.

For functions on $G$, define $(f*g)(x)=\sum_{y\in G}f(y)g(x+y)$,
the \emph{convolution} with counting measure. If $E\oplus T=G$, then
$1_E*1_T=1_G$. Taking Fourier transforms gives
\begin{equation}
 W_E(u)W_T(u)=0\qquad(u\ne0).
 \label{eq:tiling-zero-sets}
\end{equation}
The following lemma records two useful consequences.

\begin{lemma}\label{lem:transfer}
Let $E,T,\Lambda,B$ be nonempty subsets of $G$.
\begin{enumerate}
\item If $(E,\Lambda)$ is a spectral pair, $|E||T|=N$, and
$T+T\subseteq\operatorname{supp}W_\Lambda$, then $E\oplus T=G$.
\item If $E\oplus T=G$, $|B|=|E|$, and
$B+B\subseteq\operatorname{supp}W_T$, then $B$ is a spectrum for $E$.
\end{enumerate}
\end{lemma}

\begin{proof}
For the first assertion, suppose that $e+t=e'+t'$ for $e,e'\in E$ and
$t,t'\in T$. If $e\ne e'$, then $e+e'=t+t'$ is a nonzero element of
$E+E$ on which $W_\Lambda$ is nonzero, contrary to
\eqref{eq:spectral-zero-sets}. Thus $e=e'$ and $t=t'$. All $|E||T|=N$
sums are distinct, so they exhaust $G$.

For the second assertion, take distinct $b,b'\in B$.
By assumption $W_T(b+b')\ne0$, and hence
\eqref{eq:tiling-zero-sets} gives $W_E(b+b')=0$.
This says precisely that the characters indexed by $b$ and $b'$ are
orthogonal on $E$. There are $|E|$ such characters, so they form an
orthogonal basis of $L^2(E)$.
\end{proof}

A common spectrum for all tiling complements of $T$ is called a
\emph{universal spectrum} of $T$, following Lagarias and Wang
\cite{LagariasWang1997}. The second assertion of Lemma~\ref{lem:transfer}
is the standard sufficient condition for a universal spectrum
\cite{LagariasSzabo2001}.

We also record the elementary projection argument used to lift tilings.

\begin{lemma}\label{lem:projection}
Let $q:G\to V$ be a surjective linear map to a vector space over
$\mathbb{Z}_2$, and suppose that $q$ is injective on $E$.
If $q(E)\oplus K=V$, then $E\oplus q^{-1}(K)=G$.
In particular, a linear-subspace tiling complement of $q(E)$ lifts to a
linear-subspace tiling complement of $E$.
Moreover, if $E\oplus K=G$ and $K$ is a linear subspace, then
\[
 K^\perp=\{u\in G:u\cdot k=0\text{ for every }k\in K\}
\]
is a spectrum for $E$. This subspace is called the \emph{annihilator}, or orthogonal complement, of $K$.
\end{lemma}

\begin{proof}
For each $x\in G$, the unique decomposition of $q(x)$ specifies a unique
point $q(e)\in q(E)$ such that $q(x)+q(e)\in K$. Since $q$ is injective
on $E$, this also specifies a unique $e\in E$. Thus there is exactly one
representation $x=e+t$ with $e\in E$ and $t\in q^{-1}(K)$.

For the last assertion, $E$ contains exactly one representative of each
coset of $K$. The characters indexed by $K^\perp$ are constant on these
cosets. Consequently, summing their products over $E$ gives the same
orthogonality relations as summing the characters of the quotient group
$G/K$. Their number is $|G/K|=|E|$, as required.
\end{proof}

The following is a key lemma, the reason that sets of one quarter of the
ambient size are especially useful here. The two-dimensional subspace
lemma, Lemma~\ref{lem:support-plane}, concerns arbitrary indicator
functions, without any assumption of spectrality.

\begin{lemma}[two-dimensional subspace]\label{lem:support-plane}
If $\varnothing\ne A\ne G$ and $|A|\ne N/2$, then
$\operatorname{supp}W_A$ contains a two-dimensional linear subspace.
\end{lemma}

\begin{proof}
Put $a(u)=N^{-1}W_A(u)$ and $\delta=a(0)=|A|/N$. Thus
$0<\delta<1$, and multiplying the transform by $N^{-1}$ does not
change its support. Fourier inversion gives
\[
 1_A(x)=\sum_{u\in G}a(u)(-1)^{u\cdot x}.
\]
The key point is that an indicator takes only the values zero and one,
so $1_A(x)^2=1_A(x)$. Squaring the displayed sum gives
\[
\begin{aligned}
 1_A(x)^2
 &=\sum_{v,t\in G}a(v)a(t)(-1)^{(v+t)\cdot x}\\
 &=\sum_{s\in G}\left(\sum_{t\in G}a(t)a(s+t)\right)(-1)^{s\cdot x}.
\end{aligned}
\]
In the second equality we grouped terms with $v+t=s$, so that
$v=s+t$. All sums are finite. The uniqueness of the Fourier
coefficients, or the character orthogonality identity
\eqref{eq:character-orthogonality}, now gives
\[
 a(s)=\sum_{t\in G}a(t)a(s+t)\qquad(s\in G).
\]
For $s\ne0$, the distinct indices $t=0$ and $t=s$ contribute
$a(0)a(s)$ and $a(s)a(0)$, respectively. Removing these two terms and
using $a(0)=\delta$, we obtain
\begin{equation}
 (1-2\delta)a(s)=
 \sum_{t\in G\setminus\{0,s\}}a(t)a(s+t).
 \label{eq:indicator-convolution}
\end{equation}

Since $A$ is nonempty and proper, $1_A$ is not constant. Hence
$a(s)\ne0$ for some $s\ne0$: otherwise Fourier inversion would give
$1_A(x)=a(0)=\delta$ for every $x$. Fix such an $s$.
The assumption $|A|\ne N/2$ says that $1-2\delta\ne0$.
Thus the left side of \eqref{eq:indicator-convolution} is nonzero,
so at least one summand on the right is nonzero. For its index $t$,
we have $a(t)\ne0$, $a(s+t)\ne0$, and $t\notin\{0,s\}$.
Also $a(0)=\delta>0$. Therefore
\[
 \{0,s,t,s+t\}\subseteq\operatorname{supp}W_A.
\]
Over $\mathbb Z_2$, the span of the nonzero vector $s$ is precisely
$\{0,s\}$. Since $t$ is outside this span, $s,t$ are independent,
and the four displayed vectors form their two-dimensional span.
\end{proof}

The exception $|A|=N/2$ is necessary. For example, if $A$ is a
hyperplane, then $W_A$ is supported on the two-point subspace $A^\perp$,
which contains no plane. The exclusions $A=\varnothing$ and $A=G$ are
also necessary.

\begin{theorem}[quarter density]\label{thm:quarter}
Let $E$ be a nonempty spectral subset of $\mathbb{Z}_2^d$. Then
\[
 |E|\le 2^{d-2},\qquad |E|=2^{d-1},\qquad\text{or}\qquad
 E=\mathbb{Z}_2^d.
\]
If $|E|=2^{d-2}$, then $E$ has a two-dimensional linear-subspace tiling
complement. If $|E|=2^{d-1}$, then $E$ has a one-dimensional
linear-subspace tiling complement.
\end{theorem}

\begin{proof}
Let $\Lambda$ be a spectrum for $E$. Then $|\Lambda|=|E|>0$.
We distinguish the three possibilities for $\Lambda$.

First suppose that $\Lambda\ne G$ and $|\Lambda|\ne N/2$.
The two-dimensional subspace lemma, Lemma~\ref{lem:support-plane},
applies to $\Lambda$ and gives a plane
$K\subseteq\operatorname{supp}W_\Lambda$. By the symmetry of spectral
pairs, every nonzero element of $E+E$ is a zero of $W_\Lambda$.
Consequently
\[
 (E+E)\cap K=\{0\}.
\]
We claim that the translates $E+k$, $k\in K$, are pairwise disjoint.
Indeed, an equality $e+k=e'+k'$ would give
$e+e'=k+k'\in(E+E)\cap K$. The intersection condition forces
$k+k'=0$, hence $k=k'$ and $e=e'$. Since $K$ has four elements,
\[
 4|E|=\left|\bigcup_{k\in K}(E+k)\right|\le |G|=N.
\]
This proves the quarter-density bound. If $|E|=N/4$, the disjoint
translates have total size $N$ and exhaust $G$, proving $E\oplus K=G$.

Next suppose that $|\Lambda|=N/2$. The indicator of $\Lambda$ is
nonconstant, so Fourier inversion gives an $s\ne0$ with
$W_\Lambda(s)\ne0$. By \eqref{eq:spectral-zero-sets}, this $s$ cannot
belong to $E+E$. Thus $E$ and $E+s$ are disjoint: a point in their
intersection would express $s$ as the sum of two elements of $E$.
Since $2|E|=2|\Lambda|=N$, these two translates exhaust $G$.
Therefore $E\oplus\{0,s\}=G$, with a one-dimensional subspace
complement.

Finally, if $\Lambda=G$, then $|E|=|\Lambda|=N$, so $E=G$.
These cases exhaust all possibilities and prove each assertion.
\end{proof}

In dimension seven, the quarter-density theorem, Theorem~\ref{thm:quarter}, treats spectral sets of
size $32$ without using the classification of Hadamard matrices of that
order. In dimension eight, it treats spectral sets of size $64$ without
using any classification of Hadamard matrices of order $64$.
The identity \eqref{eq:indicator-convolution} and its consequence,
the quarter-density theorem, Theorem~\ref{thm:quarter}, were given in \cite{FergusonIosevich2026}.
The proof above includes the Fourier preliminaries in order to make
clear why the argument applies in every dimension.

\section{Projection and sets of small cardinality}\label{sec:small-sets}

We first record how tiling and spectrality behave under a change of
coordinates. If $M$ is an invertible linear map, the identity
\[
 ((M^T)^{-1}\lambda)\cdot Mx=\lambda\cdot x
\]
shows that $M$ sends spectral sets to spectral sets, with the indicated
change of spectrum. Translation of a set multiplies each character by a
constant of absolute value one, and hence also preserves orthogonality.
Both operations plainly preserve tiling. Translation of $A$ also multiplies $W_A(u)$ by a sign, so its support does not change. We may therefore translate a
nonempty set to contain zero whenever this is convenient.

We will also use the elementary restriction to the span of a set. Suppose
$0\in E\subseteq V\subseteq G$, where $V$ is a linear subspace. If
$E\oplus T=G$, then $E\oplus(T\cap V)=V$: a decomposition of a point
of $V$ has its first summand in $V$, and therefore its second summand in
$V$ as well. Conversely, a tiling of $V$ extends to one of $G$ by adding
a linear complement of $V$ to the translation set. In particular, a
linear-subspace tiling complement in $V$ extends to one in $G$.

Our next lemma is due to Kadowaki
\cite[Theorem 3.1 and Lemma 3.3]{Kadowaki2026}. We give the construction
in full. Its conclusion is stronger than spectrality of a particular
tiling complement: the same frequency set works for every tiling
complement of $T$.

\begin{lemma}\label{lem:universal-small}
Let $T\subseteq\mathbb Z_2^d$ have $1$, $2$, $4$, or $8$ elements.
There is a set $B\subseteq\mathbb Z_2^d$ such that
\[
 |B|=\frac{2^d}{|T|},\qquad
 B+B\subseteq\operatorname{supp}W_T.
\]
Consequently, every tiling complement of $T$ is spectral. Also, every
spectrum of $T$, when $T$ is spectral, tiles $\mathbb Z_2^d$.
\end{lemma}

\begin{proof}
Translation multiplies each Fourier coefficient of $T$ by a sign,
so it does not change its Fourier support. Suppose first that $|T|=8$.
Translate $T$ so that
$T=\{0,t_1,\ldots,t_7\}$, and put $V=\langle T\rangle$ and
$r=\dim V$. Since $T$ has eight points and contains zero, $3\le r\le7$. Write $V^*$ for the space of linear maps from $V$ to
$\mathbb Z_2$. The map
\[
 \phi:V^*\longrightarrow\mathbb Z_2^7,
 \qquad \phi(\alpha)=(\alpha(t_1),\ldots,\alpha(t_7))
\]
is injective, because $t_1,\ldots,t_7$ span $V$. Its image $U$ is
therefore an $r$-dimensional subspace. If $\operatorname{wt}(x)$ denotes
the \emph{Hamming weight} of $x$, the number of its nonzero coordinates
\cite{Hamming1950}, then
\[
 \sum_{t\in T}(-1)^{\alpha(t)}
       =8-2\operatorname{wt}(\phi(\alpha)).
\]
In particular, this sum vanishes precisely when
$\operatorname{wt}(\phi(\alpha))=4$.

Let $e_1,\ldots,e_7$ be the coordinate vectors of $\mathbb Z_2^7$, let
$j=(1,\ldots,1)$, and define
\[
 S=\{0,e_1,\ldots,e_7\},\qquad D=S\cup(j+S).
\]
The set $D$ has sixteen elements. A sum of two distinct elements in the
same part has weight one or two, and a sum of elements in different
parts has weight five, six, or seven. Thus no sum of two elements of
$D$ has weight four.

There are $2^{7-r}$ cosets of $U$. Some coset contains at least
$16/2^{7-r}=2^{r-3}$ elements of $D$. Choose exactly this many, say
$x_1,\ldots,x_{2^{r-3}}$, and translate them by $x_1$ into $U$.
Their pairwise sums still avoid weight four. Taking their inverse images
under $\phi$ gives a set $B_V\subseteq V^*$ of size $2^{r-3}$ with
\[
 \sum_{t\in T}(-1)^{(\alpha+\beta)(t)}\ne0
       \qquad(\alpha,\beta\in B_V).
\]
This includes $\alpha=\beta$, when the sum is eight.

The restriction map
\[
 \rho:\mathbb Z_2^d\longrightarrow V^*,\qquad
 \rho(u)(v)=u\cdot v,
\]
is surjective and each of its fibers has $2^{d-r}$ elements. Hence
$B=\rho^{-1}(B_V)$ has $2^{d-3}$ elements and has the required property.

For the smaller cardinalities we use the same restriction map. After
translation, a one-point or two-point set is its own span, and we take
$B_V=\{0\}$. A four-point set either is its own two-dimensional span,
when the same choice works, or has the form
$\{0,t_1,t_2,t_3\}$ with $t_1,t_2,t_3$ independent. In the latter case
take $B_V=\{0,\alpha\}$, where $\alpha(t_i)=1$ for $i=1,2,3$. The
only nonzero difference gives the character sum $1-3=-2$. Lifting
$B_V$ by restriction gives the asserted cardinality in every case.

If $E\oplus T=\mathbb Z_2^d$, the Fourier identity for tiling gives
$W_E(u)W_T(u)=0$ for $u\ne0$. Therefore distinct characters indexed
by $B$ are orthogonal on $E$, and $|B|=|E|$ makes $B$ a spectrum.
Finally, if $\Lambda$ is a spectrum of $T$, every nonzero element of
$\Lambda+\Lambda$ is a zero of $W_T$. Thus
$(\Lambda+\Lambda)\cap(B+B)=\{0\}$. The sums in $\Lambda+B$ are
distinct and number $2^d$, proving the last assertion.
\end{proof}

The following published small-set result will be our only input about
tilings of sets with at most eight elements.

\begin{theorem}[Cohen, Litsyn, Vardy, and Z\'emor]\label{lem:small-tiles}
If a nonempty set of at most eight elements tiles a vector space over
$\mathbb Z_2$, then it has a linear-subspace tiling complement.
\end{theorem}

This is established in \cite[end of Section 4]{CLVZ1996}; see also
Corollary 7.3 and Proposition 6.1 there. The ambient dimension is not
restricted. To indicate the small size of the classification involved,
translate the tile to contain zero. Its span then has dimension at most
seven, and the proof treats these possible dimensions separately. The
assertion for one, two, and four points can also be seen directly. A
four-point set is either an affine plane or, after translation, has three
independent nonzero points. In the second case, map these points to the
three nonzero points of $\mathbb Z_2^2$ and use
Lemma~\ref{lem:projection}.

\begin{corollary}\label{cor:small-spectral}
Every spectral set of cardinality $1$, $2$, $4$, or $8$ in a vector space
over $\mathbb Z_2$ has a linear-subspace tiling complement.
\end{corollary}

\begin{proof}
If $\Lambda$ is a spectrum of $E$, then $E$ is a spectrum of $\Lambda$.
Apply the last assertion of Lemma~\ref{lem:universal-small} with
$T=\Lambda$ to show that $E$ tiles, and then apply
Theorem~\ref{lem:small-tiles}.
\end{proof}

\section{Hadamard matrices and spectral cardinalities}\label{sec:hadamard}

We recall the connection with the matrices used by Ferguson and
Sothanaphan~\cite{FergusonSothanaphan2020}. A real \emph{Hadamard matrix} \cite{Hadamard1893}
of order $m$ is an $m\times m$ matrix $H$ with entries $1$ and $-1$
satisfying $HH^T=mI$. A matrix $L$ over $\mathbb Z_2$ is
\emph{log-Hadamard} if $H_{ij}=(-1)^{L_{ij}}$ is Hadamard. Equivalently,
the difference of two distinct rows of $L$ is \emph{equidistributed},
or \emph{balanced}: it contains equally many zeros and ones. We call $L$ \emph{dephased} if its first row and first column
are zero. Thus dephasing $H$ gives its usual normalization: its first
row and first column are equal to one. In this section all ranks of log-Hadamard matrices are taken
over $\mathbb Z_2$.

\subsection{An integer determinant bound}\label{subsec:cardinalities}

For $0\le j\le d$, form a list in which the integer $j$ occurs
$\binom dj$ times, in increasing order. Let $s_d(m)$ be the sum of the
first $m$ entries of this list. For a positive integer $m$, let
$\nu_2(m)$ denote the largest integer $a$ such that $2^a$ divides $m$.
Kadowaki's bound is as follows.

\begin{lemma}[Kadowaki]\label{lem:smith}
If $E\subseteq\mathbb Z_2^d$ is spectral and $m=|E|\ge2$, then
\begin{equation}\label{eq:smith-bound}
 \frac m2\nu_2(m)\ge s_d(m).
\end{equation}
\end{lemma}

\begin{proof}
Order the elements of $\mathbb Z_2^d$ by their coordinates and let
\[
 S_d=\bigl((-1)^{x\cdot y}\bigr)_{x,y\in\mathbb Z_2^d}.
\]
This is the \emph{Sylvester Hadamard matrix}, obtained by Sylvester's
tensor-product construction~\cite{Sylvester1867}: it is the $d$-fold tensor (Kronecker) product of
$S_1=\left(\begin{smallmatrix}1&1\\1&-1\end{smallmatrix}\right)$.
The explicit integer operations
\[
 \begin{pmatrix}1&0\\1&-1\end{pmatrix}
 S_1
 \begin{pmatrix}1&-1\\0&1\end{pmatrix}
 =\begin{pmatrix}1&0\\0&2\end{pmatrix}
\]
have integer inverses. Taking their $d$-fold tensor products expresses
$S_d=ADB$, where $A,B$ are invertible integer matrices with integer
inverses, and $D$ is diagonal with $2^j$ occurring $\binom dj$ times.
This is the \emph{Smith normal form} of $S_d$, introduced by Smith
\cite{Smith1861}; the displayed factorization is all that we need.

Apply the Cauchy--Binet formula to any $m\times m$ minor of $ADB$.
Each summand is an integer times a product of $m$ distinct diagonal
entries of $D$. Every such product is divisible by $2^{s_d(m)}$.
Consequently every $m\times m$ minor of $S_d$ is divisible by this
power of two.

If $\Lambda$ is a spectrum of $E$, the submatrix
$H=((-1)^{\lambda\cdot e})_{\lambda\in\Lambda,e\in E}$ satisfies
$HH^T=mI$. Taking determinants gives $(\det H)^2=m^m$, and hence
\[
 \nu_2(|\det H|)=\frac m2\nu_2(m).
\]
The divisibility just proved gives \eqref{eq:smith-bound}. This is the
argument of \cite[Section 4]{Kadowaki2026}, with the integer
factorization written explicitly.
\end{proof}

\begin{corollary}\label{cor:cardinality-eight}
Every spectral subset of $\mathbb Z_2^7$ or $\mathbb Z_2^8$ has
cardinality a power of two.
\end{corollary}

\begin{proof}
A Hadamard matrix has order one, two, or a multiple of four. Indeed,
normalize its first row to consist of ones. If there is a second row,
orthogonality gives equally many signs of each kind. If there is a
third row, orthogonality to the first two rows says that within each
half selected by the second row the third row has equally many signs
of each kind. The order is therefore divisible by four.

In dimension eight, the quarter-density theorem, Theorem~\ref{thm:quarter}, reduces the cardinalities
to those at most $64$, together with $128$ and $256$. For the multiples
of four at most $64$ which are not powers of two, the two sides of
\eqref{eq:smith-bound} are
\[
\begin{array}{c|rrrrrrrrrrr}
m&12&20&24&28&36&40&44&48&52&56&60\\ \hline
\frac m2\nu_2(m)&12&20&36&28&36&60&44&96&52&84&60\\
s_8(m)&14&30&38&46&62&73&85&97&109&121&133.
\end{array}
\]
Every column contradicts the bound. The entries are obtained by adding
the list consisting of one zero, eight ones, twenty-eight twos, and
fifty-six threes; no matrix classification is involved.

In dimension seven, the quarter-density theorem, Theorem~\ref{thm:quarter}, leaves only the four
non-power-of-two possibilities $12,20,24,28$. Here the corresponding
values of $s_7(m)$ are $15,31,39,47$, which are again larger than
$12,20,36,28$. This proves both assertions.
\end{proof}

\subsection{Canonical point sets}\label{subsec:canonical}

Let $L$ be a dephased log-Hadamard matrix of rank $r$. Choose a matrix
$R$ whose rows are a basis of the row space of $L$. We call the set of
columns of $R$ a \emph{canonical column set} of $L$. The word canonical
refers to the set up to an invertible linear map: changing the chosen
row basis applies the same invertible map to every column. These
columns are distinct. Otherwise two columns of $L$ would be equal,
contrary to Hadamard orthogonality.

\begin{lemma}[Kadowaki]\label{lem:canonical}
Let $(E,\Lambda)$ be a spectral pair in $\mathbb Z_2^d$, with both
sets translated to contain zero. Put
$L=(\lambda\cdot e)_{\lambda\in\Lambda,e\in E}$ and $r=\operatorname{rank}L$.
Then $r\le d$, and there is a surjective linear map
$q:\mathbb Z_2^d\to\mathbb Z_2^r$ which maps $E$ bijectively onto a
canonical column set of $L$.
\end{lemma}

\begin{proof}
The matrix $L$ is a product of a matrix with $d$ columns and a matrix
with $d$ rows, so its rank is at most $d$. Choose
$\lambda_1,\ldots,\lambda_r$ whose rows of $L$ form a row-space basis,
and define
\[
 q(x)=(\lambda_1\cdot x,\ldots,\lambda_r\cdot x).
\]
The $\lambda_i$ are linearly independent: a relation between these
vectors would give the same relation between their rows of $L$.
Thus $q$ is surjective. Its values on $E$ are exactly the distinct
columns of the selected row-basis matrix.
\end{proof}

Two Hadamard matrices are called \emph{equivalent} if one is obtained
from the other by permuting rows and columns and multiplying rows and
columns by $-1$. Transposition is not included in this definition.

\begin{lemma}\label{lem:equivalence}
Equivalent Hadamard matrices have affinely equivalent canonical column
sets after dephasing. In particular, their dephased log-Hadamard
matrices have the same rank. Transposition exchanges the roles of the
canonical column and row sets.
\end{lemma}

\begin{proof}
Row and column sign changes cancel when the matrices are dephased.
Apart from relabeling, a permutation only changes the row and column
chosen for dephasing. If these are $i_0,j_0$, the new phase matrix has
entries
\[
 L'_{ij}=L_{ij}+L_{ij_0}+L_{i_0j}+L_{i_0j_0}.
\]
Take a rank factorization $L=PR$, with both factors of rank $r$, and
write $p_i$ for the rows of $P$ and $c_j$ for the columns of $R$.
Because $L$ is dephased and both factors have full rank, $p_0=c_0=0$.
Consequently
\[
 L'_{ij}=(p_i+p_{i_0})\cdot(c_j+c_{j_0}).
\]
The translated row family still spans $\mathbb Z_2^r$: it contains
$p_{i_0}$, coming from $i=0$, and hence contains every original row
in its span. The same applies to the translated column family. The
new factorization therefore still has rank $r$. Selecting a basis of
its row space changes the translated columns by an invertible linear
map. This proves the asserted affine equivalence. Transposing $L=PR$
exchanges the two families, proving the last assertion.
\end{proof}

The rank assertion also follows from the dephasing theorem in
\cite[Proposition 4.3]{FergusonSothanaphan2020}. The stronger statement
about the point sets is \cite[Lemma 5.3]{Kadowaki2026}; the proof above
also explains why a classification taken up to transposition requires
both orientations to be checked.

\begin{lemma}\label{lem:small-spectral}
Every spectral set of cardinality $1$, $2$, $4$, $8$, or $16$ in a vector
space over $\mathbb Z_2$ has a linear-subspace tiling complement.
\end{lemma}

\begin{proof}
Cardinalities at most eight follow from
Corollary~\ref{cor:small-spectral}. For cardinality sixteen, use Hall's
classification~\cite{Hall1961} of the five ordinary equivalence classes
of Hadamard matrices of order sixteen. We give all five representatives
in Appendix~\ref{app:order16}, together with five column indices
$c_0,c_1,c_2,c_3,c_4$ for each representative $H$. In each case the
sixteen relative sign patterns
\begin{equation}\label{eq:order16-pattern}
 \bigl(H_{i,c_1}H_{i,c_0},H_{i,c_2}H_{i,c_0},
       H_{i,c_3}H_{i,c_0},H_{i,c_4}H_{i,c_0}\bigr),
       \qquad 0\le i<16,
\end{equation}
are all distinct. The appendix displays these patterns, so this finite
verification can be made directly from the printed matrices.

The printed representatives have their first row and first column equal to one, so their phase matrices are dephased. To see what the verification means, represent the rows as points $x_i$
and the columns as linear functionals $a_j$, so that
$H_{ij}=(-1)^{a_j(x_i)}$. The four sign coordinates in
\eqref{eq:order16-pattern} are the characters associated with
$a_{c_j}+a_{c_0}$, for $1\le j\le4$. Distinctness says that their four
values give a linear map which is bijective from the row point set
onto $\mathbb Z_2^4$. Its kernel is a tiling complement by
Lemma~\ref{lem:projection}.

Lemma~\ref{lem:equivalence} transfers this property to the canonical
row point set of every Hadamard matrix of order sixteen. Applied to
the transpose, it gives the same property for every canonical column
set. There is no additional transposition assumption here: Hall's five
classes are ordinary equivalence classes, so the transpose belongs to
one of them. Finally, Lemmas~\ref{lem:canonical} and
\ref{lem:projection} lift a subspace complement of the canonical
column set to one of the original set. This is the subspace form of
the order-sixteen argument in \cite[Proposition 6.1]{Kadowaki2026}.
\end{proof}

\section{A shorter proof in dimension seven}\label{sec:seven}

Kadowaki~\cite{Kadowaki2026} first established the equivalence of
spectrality and tiling in dimension seven. The proof below also gives
a linear-subspace complement. His proof of the spectral-to-tiling implication includes an examination
of Hadamard matrices of order thirty-two. The quarter-density theorem, Theorem~\ref{thm:quarter},
makes that examination unnecessary. We retain his elementary
cardinality bound, projection argument, and construction for small
tiling partners, all of which were proved above.

\begin{theorem}\label{thm:seven}
For a nonempty subset $E$ of $\mathbb Z_2^7$, the following are
equivalent: $E$ is spectral; $E$ tiles $\mathbb Z_2^7$; and $E$ has a
linear-subspace tiling complement.
\end{theorem}

\begin{proof}
Suppose first that $E$ is spectral. By
Corollary~\ref{cor:cardinality-eight}, its cardinality is a power of
two. If $|E|\le16$, Lemma~\ref{lem:small-spectral} gives a
linear-subspace tiling complement. If $|E|=32$ or $64$,
the quarter-density theorem, Theorem~\ref{thm:quarter}, gives one of dimension two or one,
respectively. If $|E|=128$, the complement is $\{0\}$.

Now suppose that $E$ tiles. Its cardinality divides $128$, so it is
again a power of two. If $|E|\le8$, Theorem~\ref{lem:small-tiles}
gives a subspace complement and Lemma~\ref{lem:projection} makes
$E$ spectral. If $|E|\ge16$, any tiling complement has at most eight
elements; Lemma~\ref{lem:universal-small} makes $E$ spectral in this
case as well. The first part of the proof now gives a subspace
complement. Finally, a set with a subspace tiling complement plainly
tiles and is spectral by Lemma~\ref{lem:projection}.
\end{proof}

Apart from the published small-tile theorem and Hall's five-class
classification, this proof uses only the arguments given in the paper.
The additional checks attached to Hall's classification consist of the
five lists of sixteen patterns in Appendix~\ref{app:order16}. In
particular, neither the classification of Hadamard matrices of order
thirty-two nor the correctness of the dimension-seven computations in
\cite{Kadowaki2026} is needed for this proof.

\section{The eight-dimensional case}
\label{sec:eight}

We now prove the remaining case of Theorem~\ref{thm:main}. The quarter-density theorem, Theorem~\ref{thm:quarter}, treats spectral sets of size 64, and the preceding section treats spectral sets of size at most 16 and tiling complements with at most eight points. Two cases remain: spectral sets with 32 points, and tiles with 16 points. We discuss them separately.

\subsection{Spectral sets with thirty-two points}

\begin{lemma}\label{lem:order32}
Every canonical column set of a Hadamard matrix of order 32 whose dephased log-Hadamard matrix has rank at most eight over $\mathbb Z_2$ has a linear-subspace tiling complement.
\end{lemma}

\begin{proof}
We use the complete classification in \cite{KharaghaniTayfehRezaie2013}, with the representatives distributed by McKay \cite{McKayData}. There are $6{,}857{,}010$ classes when a matrix and its transpose are identified. We refer to these as \emph{extended classes}. Dephasing a representative of each class and calculating its rank over $\mathbb Z_2$ leaves the following classes of rank at most eight:
\[
\begin{array}{c|rrrr|r}
\text{Rank}&5&6&7&8&\text{Total}\\ \hline
\text{Number of extended classes}&1&2&13&150&166.
\end{array}
\]
These counts were already obtained by Kadowaki \cite[Table~6.2]{Kadowaki2026}. The point-set equivalence lemma shows that we must check both orientations of each representative. This gives 332 canonical sets.

For each resulting set $A\subset\mathbb Z_2^r$, where $|A|=32$ and $5\le r\le8$, a certificate consists of $r-5$ independent vectors spanning a subspace $K$ for which
\begin{equation}\label{eq:order32certificate}
 \bigl|\{a+k:a\in A,\ k\in K\}\bigr|=2^r.
\end{equation}
There are exactly $|A||K|=32\cdot2^{r-5}=2^r$ pairs in this sum. Hence the displayed equality proves that every element has precisely one representation, or $A\oplus K=\mathbb Z_2^r$.

The accompanying programs produce and check such a certificate for each of the 332 sets. They also verify the Hadamard equations, reconstruct the point sets from the matrix records, and check the extraction from the complete classification. The input conventions, coverage checks, and reproduction commands are given in Appendix~\ref{app:computations}. Thus the finite verification proves the assertion for every allowed representative. Lemma~\ref{lem:equivalence} extends it to every matrix in the assertion.
\end{proof}

\begin{corollary}\label{cor:spectral-eight}
Every spectral subset of $\mathbb Z_2^8$ has a linear-subspace tiling complement.
\end{corollary}

\begin{proof}
By Corollary~\ref{cor:cardinality-eight}, its size is a power of two. Lemma~\ref{lem:small-spectral} treats sizes at most 16. For size 32, project onto a canonical column set by Lemma~\ref{lem:canonical}, apply Lemma~\ref{lem:order32}, and lift its subspace complement by Lemma~\ref{lem:projection}. The quarter-density theorem, Theorem~\ref{thm:quarter}, treats sizes 64 and 128. The whole group has the complement $\{0\}$.
\end{proof}

\subsection{Tiles with sixteen points}

An arbitrary tiling complement need not be a subspace. For this reason, we cannot assume at the outset that a sixteen-point tile is a graph on a four-dimensional subspace. We obtain that conclusion by relating a tiling to a perfect code and then passing to its translation kernel.

For a vector $x\in\mathbb Z_2^n$, its \emph{Hamming weight}, denoted by $\operatorname{wt}(x)$, is the number of nonzero coordinates. The \emph{Hamming distance} between $x$ and $y$ is $\operatorname{wt}(x+y)$ \cite{Hamming1950}. A \emph{perfect binary one-error-correcting code} of length 15 is a set $C\subset\mathbb Z_2^{15}$ such that the Hamming balls of radius one centered at its points partition the space. If
\[
 S=\{0,e_1,\ldots,e_{15}\},
\]
where the $e_i$ are the coordinate vectors, the perfect-code condition is exactly $S\oplus C=\mathbb Z_2^{15}$. It implies $|C|=2048$ and that distinct codewords have distance at least three. Codes are regarded as equivalent when one can be obtained from the other by a coordinate permutation and a translation.

\begin{definition}
The \emph{translation kernel} of a set $C\subset\mathbb Z_2^{15}$ is
\[
 M(C)=\{u\in\mathbb Z_2^{15}:C+u=C\}.
\]
\end{definition}

This is a linear subspace. If $0\in C$, then $M(C)\subset C$, since $C+u=C$ implies $u=0+u\in C$. Thus, for a perfect code containing zero, no nonzero vector of its translation kernel has weight one or two.

\begin{lemma}\label{lem:code-certificates}
Let $C$ be a perfect code of length 15 whose translation kernel $M$ has dimension at least seven. The image of $S=\{0,e_1,\ldots,e_{15}\}$ in $\mathbb Z_2^{15}/M$ has a linear-subspace tiling complement.
\end{lemma}

\begin{proof}
The classification of \"Osterg\aa rd and Pottonen \cite{OstergardPottonen2009} contains 5983 equivalence classes. Precisely 69 have translation-kernel dimension at least seven, with distribution
\[
\begin{array}{c|rrrr|r}
\dim M&7&8&9&11&\text{Total}\\ \hline
\text{Number of classes}&48&17&3&1&69.
\end{array}
\]
The kernel distribution is also recorded in \cite{OstergardPottonenPhelps2010}.

For each of these 69 representatives, the accompanying certificate gives $11-\dim M$ independent vectors in the quotient which span a tiling complement of the image of $S$. The verifier checks the perfect-code partition, computes the translation kernel, forms the quotient map, and checks every sum in the asserted tiling. In particular, it does not infer a translation kernel by treating a nonlinear code as a linear code. The complete-source extraction and the quotient certificates are described in Appendix~\ref{app:computations}.

These checks suffice for all codes, rather than only for the chosen representatives. Translating $C$ leaves $M(C)$ unchanged. A coordinate permutation sends $M(C)$ to its permuted subspace and preserves $S$. It therefore induces a linear equivalence of the corresponding quotient point sets. Such an equivalence preserves the existence of a subspace complement.
\end{proof}

\begin{lemma}\label{lem:sixteen-tiles}
Every sixteen-point tile of affine dimension at most eight in a vector space over $\mathbb Z_2$ has a linear-subspace tiling complement.
\end{lemma}

\begin{proof}
Translate the tile $E$ so that $0\in E$, and let $V$ be its linear span. Write $r=\dim V\le8$. If $E\oplus T=G$, then
\[
 E\oplus(T\cap V)=V.
\]
Indeed, in the unique representation of a point of $V$, its summand from $E$ belongs to $V$, so its other summand belongs to $V$ as well. Put $T_V=T\cap V$.

List the fifteen nonzero points of $E$ as $a_1,\ldots,a_{15}$, and define a linear map
\[
 H:\mathbb Z_2^{15}\longrightarrow V,\qquad H(e_i)=a_i.
\]
The map is surjective because $E$ spans $V$, and it maps $S$ bijectively onto $E$. Put $C=H^{-1}(T_V)$. We claim that $S\oplus C=\mathbb Z_2^{15}$. For $x\in\mathbb Z_2^{15}$, the unique decomposition of $H(x)$ in $E\oplus T_V$ specifies a unique $s\in S$ with $H(x+s)\in T_V$. Thus $x+s\in C$, and $x=s+(x+s)$ is the required unique representation. This proves the claim, so $C$ is a perfect code.

By translating $T_V$ we may arrange that $0\in T_V$, without changing the tiling property; then $0\in C$. Put $K_0=\ker H$ and $M=M(C)$. Every vector of $K_0$ preserves $C=H^{-1}(T_V)$, so
\[
 K_0\subset M,\qquad \dim M\ge\dim K_0=15-r\ge7.
\]
Since distinct points of $C$ have distance at least three and $M\subset C$, the quotient map modulo $M$ is injective on $S$: the sum of two distinct points of $S$ has weight one or two.

Under the identification $V=\mathbb Z_2^{15}/K_0$, we therefore have a surjective linear map
\[
 q:V\longrightarrow\mathbb Z_2^{15}/M
\]
which is injective on $E$ and takes $E$ onto the image of $S$. Lemma~\ref{lem:code-certificates} supplies a subspace complement of this image. Its inverse image under $q$ is a subspace complement of $E$ in $V$, by Lemma~\ref{lem:projection}. Adding any linear complement of $V$ in $G$ gives a subspace complement in $G$. Undoing the translation of $E$ preserves this complement.
\end{proof}

\begin{proof}[Proof of Theorem~\ref{thm:main} in dimension eight]
Corollary~\ref{cor:spectral-eight} proves that a spectral set has a subspace tiling complement. A set with a subspace tiling complement is spectral by Lemma~\ref{lem:projection}, and is plainly a tile.

It remains to prove that every tile $E$ is spectral. Its cardinality divides 256. If $|E|\le8$, Theorem~\ref{lem:small-tiles} gives a subspace complement. If $|E|=16$, Lemma~\ref{lem:sixteen-tiles} does so. For $|E|=32,64,128$, or 256, its tiling complement has at most eight points, and Lemma~\ref{lem:universal-small} shows that $E$ is spectral. This exhausts the possibilities.
\end{proof}

We can now identify the two uses of the larger classifications precisely. The order-32 classification enters only through Lemma~\ref{lem:order32}, and the perfect-code classification enters only through Lemma~\ref{lem:code-certificates}. Neither is used in the seven-dimensional proof. Kadowaki's cardinality, projection, and common-spectrum arguments have been proved in the preceding section, so the eight-dimensional proof does not depend on the acceptance of his dimension-seven theorem.

\section{Why the Fourier-support argument does not extend to one eighth}
\label{sec:eighth}

The quarter-density theorem, Theorem~\ref{thm:quarter}, suggests a possible approach to spectral sets
of size $64$ in $\mathbb{Z}_2^9$. If the Fourier support of a spectrum
for such a set contained a three-dimensional subspace, then
Lemma~\ref{lem:transfer} would give a tiling complement of size eight.
We might therefore hope that the Fourier support of every spectral set
of size $64$ in $\mathbb{Z}_2^9$ contains such a subspace.
The following example shows that this is not the case.
It gives an elementary presentation of the obstruction in
\cite{FergusonIosevich2026}; no computer calculation is needed for its
verification.

\begin{proposition}\label{prop:eighth-counterexample}
There is a spectral set $B\subseteq\mathbb{Z}_2^9$ of size $64$ whose
Fourier support contains no three-dimensional linear subspace.
Nevertheless, $B$ has a linear-subspace tiling complement.
\end{proposition}

\begin{proof}
Let $e_1,\ldots,e_6$ be the standard basis of $\mathbb{Z}_2^6$, put
$j=e_1+\cdots+e_6$, and define
\[
 C=\{0,e_1,e_2,e_3,e_4,e_5,e_6,j\}.
\]
We first show that $C$ is spectral. Let $v_1,\ldots,v_6$ be any six
of the seven nonzero vectors of $\mathbb{Z}_2^3$, and let $v_7$ be the
remaining nonzero vector. Define the linear map
$q:\mathbb{Z}_2^6\to\mathbb{Z}_2^3$ by $q(e_i)=v_i$.
The sum of all seven nonzero vectors of $\mathbb{Z}_2^3$ is zero,
because each coordinate is one four times. Thus
\[
 q(j)=v_1+\cdots+v_6=v_7.
\]
Consequently $q$ maps $C$ bijectively onto $\mathbb{Z}_2^3$.
It follows that $C\oplus\ker q=\mathbb{Z}_2^6$ and that $C$ is
spectral, by Lemma~\ref{lem:projection}.

For $u\in\mathbb{Z}_2^6$, let $\operatorname{wt}(u)$ be its number of
nonzero coordinates. Writing $w=\operatorname{wt}(u)$, we have
\begin{equation}
 W_C(u)=1+\sum_{i=1}^6(-1)^{u_i}+(-1)^{u_1+\cdots+u_6}
       =7-2w+(-1)^w.
 \label{eq:eighth-walsh}
\end{equation}
For $w=0,1,\ldots,6$, these values are respectively
\[
 8,\quad4,\quad4,\quad0,\quad0,\quad-4,\quad-4.
\]
Hence
\begin{equation}
 \operatorname{supp}W_C
 =\{u:\operatorname{wt}(u)\in\{0,1,2,5,6\}\}=C+C.
 \label{eq:eighth-support}
\end{equation}
The last equality follows directly by adding two points of $C$:
the nonzero sums are $e_i$, $e_i+e_k$ with $i\ne k$, $j+e_i$, and $j$.

We next show that $C+C$ contains no three-dimensional subspace.
First observe that the only nonempty even-sized subset of $C$ whose
sum is zero is $C$ itself. Indeed, a relation among
$e_1,\ldots,e_6,j$ either uses none of these vectors or uses all seven:
comparing the six coordinates proves this assertion. Including the
zero vector changes the number of selected points but not their sum.
It follows in particular that every nonzero point of $C+C$ has a
unique expression as the sum of an unordered pair of distinct points
of $C$.

It is useful to regard such a pair as an edge joining two points of
$C$. Suppose that three distinct nonzero points of $C+C$ sum to zero.
Take their three edges. The vertices incident to an odd number of
these edges form an even-sized subset of $C$ with sum zero. There are
at most six such vertices, so the preceding observation shows that
there are none. The three distinct edges therefore form a triangle.
Thus every two-dimensional subspace contained in $C+C$ has, for its
three nonzero points, the three edges of a triangle on $C$.

Suppose now that a three-dimensional subspace $U$ were contained in
$C+C$. Choose a plane $P\subseteq U$ and an $x\in U\setminus P$.
The three edges for $P\setminus\{0\}$ form a triangle. For each
$y\in P\setminus\{0\}$, the three points $x,y,x+y$ are distinct and
nonzero and belong to $C+C$. Their edges also form a triangle.
In particular, the edge for $x$ must meet each of the three edges of
the original triangle. The only edges that meet all three are the
three edges of that triangle itself. This would put $x$ in $P$,
a contradiction.

Finally, take
\[
 B=C\times\mathbb{Z}_2^3\subseteq
 \mathbb{Z}_2^6\times\mathbb{Z}_2^3=\mathbb{Z}_2^9.
\]
The map $(x,y)\mapsto(q(x),y)$ maps $B$ bijectively onto
$\mathbb{Z}_2^6$. Its kernel is a three-dimensional tiling complement,
so $B$ is spectral and $|B|=8\cdot8=64$.
For $u\in\mathbb{Z}_2^6$ and $v\in\mathbb{Z}_2^3$, its transform is
\[
 W_B(u,v)=W_C(u)\sum_{y\in\mathbb{Z}_2^3}(-1)^{v\cdot y}
 =\begin{cases}
 8W_C(u),&v=0,\\
 0,&v\ne0.
 \end{cases}
\]
Thus $\operatorname{supp}W_B=(C+C)\times\{0\}$, which contains no
three-dimensional subspace by the argument above.
\end{proof}

Notice that this is not a counterexample to Fuglede's conjecture.
It shows that a subspace tiling complement need not be found inside
the Fourier support of a specified spectrum by extending the proof of
the quarter-density theorem, Theorem~\ref{thm:quarter}. In fact, the same construction already gives
a spectral set of density one eighth in dimension six. Taking its
product with $\mathbb{Z}_2^3$ places the obstruction in the
dimension-nine setting of interest here.

\section{Further questions}
\label{sec:questions}

The remaining dimension over $\mathbb Z_2$ between the positive results proved here and the counterexample of Ferguson and Sothanaphan is dimension nine. The broader preliminary preprint \cite{FergusonIosevich2026} gives further reductions in that dimension, including a reduction of the spectral-to-tiling implication to spectral sets of size 64. Those additional arguments are not needed in the present paper. The following question isolates the case for which the quarter-density theorem, Theorem~\ref{thm:quarter}, no longer gives a tiling.

\begin{question}
Does every spectral subset of $\mathbb Z_2^9$ with 64 elements tile? More strongly, does every such set have a linear-subspace tiling complement?
\end{question}

Proposition~\ref{prop:eighth-counterexample} does not answer either question in the negative. Its set is itself a tile with a subspace complement. What the example rules out is a statement about the Fourier support of every spectrum, rather than the possibility of choosing a suitable spectrum. This distinction gives an equivalent formulation of the stronger question.

\begin{question}
If $E\subset\mathbb Z_2^9$ is spectral and $|E|=64$, must there be a spectrum $\Lambda$ for $E$ whose Fourier support contains a three-dimensional subspace?
\end{question}

If such a spectrum exists, Lemma~\ref{lem:transfer} gives a subspace tiling complement of $E$. Conversely, if $E\oplus K=\mathbb Z_2^9$ with $\dim K=3$, then $K^\perp$ is a spectrum of $E$ and its Fourier transform is supported precisely on $K$. Thus this formulation is indeed equivalent to the stronger question. A specified spectrum need not work, as our example shows.

There is also a natural question about the two finite verifications in dimension eight. The individual certificates are short, but their coverage uses two substantial published classifications.

\begin{question}
Can the subspace-complement assertions of Lemmas~\ref{lem:order32} and~\ref{lem:sixteen-tiles} be proved without using the complete classifications of Hadamard matrices of order 32 and perfect codes of length 15?
\end{question}

An argument using only the rank bound in the first case, or the existence of a translation kernel of dimension at least seven in the second, would give a more direct explanation of the finite results. The quarter-density theorem, Theorem~\ref{thm:quarter}, supplies such an explanation for the spectral cardinality 64, and for cardinality 32 in dimension seven. The example in Section~\ref{sec:eighth} explains one obstruction to obtaining the same explanation at the next density.

\section*{Acknowledgments}
The authors acknowledge collaboration with OpenAI's ChatGPT in preparing this draft, developing and checking the computer programs, and simplifying the presentation of the proofs.

\appendix

\section{The five matrices of order sixteen}\label{app:order16}

We give all the data for Lemma~\ref{lem:small-spectral}. The five
matrices are McKay's representatives~\cite{McKayData} of Hall's five
ordinary equivalence classes~\cite{Hall1961}. The completeness of this
list is the published classification; the calculation here only verifies
an additional property of each representative.

In Table~\ref{tab:order16-matrices}, a string gives one row of the
corresponding Hadamard matrix, with $1$ denoting $+1$ and $0$ denoting
$-1$. Rows and columns are indexed from zero, and the leftmost digit is
column zero. Thus the first row of each matrix consists entirely of
$+1$'s. This convention is only for displaying the matrices; these
strings are not the rows of their log-Hadamard matrices.

\begin{table}[htbp]
\centering
\begingroup
\setlength{\tabcolsep}{3pt}
\scriptsize
\begin{tabular}{r|lllll}
Row&$H_0$&$H_1$&$H_2$&$H_3$&$H_4$\\ \hline
0&\texttt{1111111111111111}&\texttt{1111111111111111}&\texttt{1111111111111111}&\texttt{1111111111111111}&\texttt{1111111111111111}\\
1&\texttt{1111111100000000}&\texttt{1111111100000000}&\texttt{1111111100000000}&\texttt{1111111100000000}&\texttt{1111111100000000}\\
2&\texttt{1111000011110000}&\texttt{1111000011110000}&\texttt{1111000011110000}&\texttt{1111000011110000}&\texttt{1111000011110000}\\
3&\texttt{1111000000001111}&\texttt{1111000000001111}&\texttt{1111000000001111}&\texttt{1111000000001111}&\texttt{1111000000001111}\\
4&\texttt{1100110011001100}&\texttt{1100110011001100}&\texttt{1100110011001100}&\texttt{1100110011001100}&\texttt{1100110011001100}\\
5&\texttt{1100110000110011}&\texttt{1100110000110011}&\texttt{1100110000110011}&\texttt{1100110000110011}&\texttt{1100110000110011}\\
6&\texttt{1100001111000011}&\texttt{1100001111000011}&\texttt{1100001111000011}&\texttt{1100001111000011}&\texttt{1100001110101010}\\
7&\texttt{1100001100111100}&\texttt{1100001100111100}&\texttt{1100001100111100}&\texttt{1100001100111100}&\texttt{1100001101010101}\\
8&\texttt{1010101010101010}&\texttt{1010101010101010}&\texttt{1010101010101010}&\texttt{1010101010101010}&\texttt{1010101011000011}\\
9&\texttt{1010101001010101}&\texttt{1010101001010101}&\texttt{1010101001010101}&\texttt{1010100110010101}&\texttt{1010101000111100}\\
10&\texttt{1010010110100101}&\texttt{1010010110100101}&\texttt{1010010110011001}&\texttt{1010011001011001}&\texttt{1010010110011001}\\
11&\texttt{1010010101011010}&\texttt{1010010101011010}&\texttt{1010010101100110}&\texttt{1010010101100110}&\texttt{1010010101100110}\\
12&\texttt{1001100110011001}&\texttt{1001100110010110}&\texttt{1001100110010110}&\texttt{1001101001010110}&\texttt{1001100110010110}\\
13&\texttt{1001100101100110}&\texttt{1001100101101001}&\texttt{1001100101101001}&\texttt{1001100101101001}&\texttt{1001100101101001}\\
14&\texttt{1001011010010110}&\texttt{1001011010011001}&\texttt{1001011010100101}&\texttt{1001011010100101}&\texttt{1001011010100101}\\
15&\texttt{1001011001101001}&\texttt{1001011001100110}&\texttt{1001011001011010}&\texttt{1001010110011010}&\texttt{1001011001011010}\\
\end{tabular}
\endgroup
\caption{Representatives of the five classes of Hadamard matrices of order sixteen.}
\label{tab:order16-matrices}
\end{table}

For each matrix, Table~\ref{tab:order16-patterns} specifies the five
columns used in \eqref{eq:order16-pattern}. It also displays the sixteen
patterns, in row order. Here a sign pattern $(s_1,s_2,s_3,s_4)$ is
represented by the integer
\[
 8\frac{1-s_1}{2}+4\frac{1-s_2}{2}
       +2\frac{1-s_3}{2}+\frac{1-s_4}{2}.
\]
We write these integers as hexadecimal digits, using $A$ through $F$
for $10$ through $15$. Every row of the pattern table contains each of the
sixteen digits exactly once. This is the required bijection.

\begin{table}[htbp]
\centering
\small
\begin{tabular}{c|c|l}
Matrix&$(c_0,c_1,c_2,c_3,c_4)$&Patterns in row order\\ \hline
$H_0$&$(0,1,2,4,8)$&\texttt{0 1 2 3 4 5 6 7 8 9 A B C D E F}\\
$H_1$&$(0,1,2,8,12)$&\texttt{0 3 1 2 4 7 5 6 8 B 9 A D E C F}\\
$H_2$&$(0,1,8,10,12)$&\texttt{0 7 1 6 2 5 3 4 8 F A D B C 9 E}\\
$H_3$&$(0,1,8,10,12)$&\texttt{0 7 1 6 2 5 3 4 8 B E D F C 9 A}\\
$H_4$&$(0,8,9,10,12)$&\texttt{0 F 1 E 2 D 4 B 3 C 6 9 7 8 5 A}\\
\end{tabular}
\caption{Five projection certificates for the matrices in Table~\ref{tab:order16-matrices}.}
\label{tab:order16-patterns}
\end{table}

The standalone program \texttt{anc/verify\_order16.py} contains these
same matrices and column choices. It checks row and column
orthogonality, reconstructs all eighty patterns, and checks the
corresponding dephased ranks $4,5,6,7,7$. If the supplied copy of
McKay's \texttt{had16.txt} is present, it also compares every entry with
the embedded data. The command is
\begin{verbatim}
python3 anc/verify_order16.py
\end{verbatim}
Only Python 3 and its standard library are required. The mathematical
checks are explicit conditional statements and remain active when
Python is run with optimization enabled. The optional program
\texttt{anc/verify\_cardinalities.py} recomputes the integer tables in
Corollary~\ref{cor:cardinality-eight}; those tables can also be checked
by the displayed additions alone.

\section{The finite computations}
\label{app:computations}

The computer-assisted part of the proof uses two published classifications.
The first is the classification of Hadamard matrices of order $32$ by
Kharaghani and Tayfeh-Rezaie \cite{KharaghaniTayfehRezaie2013}, in the
representation distributed by McKay \cite{McKayData}. The second is the
classification of perfect binary one-error-correcting codes of length $15$
by {\"O}sterg{\aa}rd and Pottonen \cite{OstergardPottonen2009}.
We use the completeness of these classifications as published mathematical
results. The programs supplied with this paper verify the extraction of
the relevant records and the assertions we make about them. They do not
reconstruct either classification.

All calculations use integer arithmetic. A vector of
$\mathbb Z_2^r$ is represented by an integer between $0$ and $2^r-1$;
its $j$th coordinate is the coefficient of $2^j$. Addition is therefore
bitwise exclusive-or. In particular, no numerical tolerance is involved
in deciding orthogonality, rank, or whether a set tiles.

\subsection{Hadamard matrices of order 32}

The source file \texttt{had32\_tr.txt.gz} contains $6{,}857{,}010$
representatives, with transposition included in the equivalence relation.
Every line contains $32$ hexadecimal integers, one for each row. A set bit
represents $+1$ and an unset bit represents $-1$. We use the compressed
file with SHA-256 digest
\begin{verbatim}
e51a18c458e1bae995e335e5067b53078e7e8c6901278e755d33c1b281547520
\end{verbatim}
The program \texttt{hadamard\_scan.cpp} reads every record and selects those
whose dephased log-Hadamard matrix has rank at most eight. It has two modes.
The first eliminates rows using the highest nonzero coordinate as pivot.
The second forms the dephased columns directly from the signs and eliminates
them using the lowest nonzero coordinate as pivot. In the second calculation
one may stop as soon as the rank exceeds eight. The two complete scans give
exactly the same records, not merely the same numbers of records. Their
common result is
\[
\begin{array}{c|rrrr|r}
\text{rank}&5&6&7&8&\text{total}\\ \hline
\text{number of representatives}&1&2&13&150&166.
\end{array}
\]

For each representative, both the matrix and its transpose are considered.
The program \texttt{hadamard\_verify.py} first checks the Hadamard equations
by counting the positions in which each pair of rows differs. It then
constructs a canonical column set $A\subseteq\mathbb Z_2^r$.
For each of the resulting $332$ sets, the file
\texttt{hadamard\_certificates.json} gives a basis of a subspace $K$ of
dimension $r-5$. The verifier checks independence of this basis and counts
every sum $a+k$. Each element of $\mathbb Z_2^r$ occurs exactly once.
This proves $A\oplus K=\mathbb Z_2^r$.

The connection between a certificate and its source matrix is also checked.
The coordinate rows of the supplied point list must have rank $r$ and,
together with the source matrix, must still have rank $r$. Thus they span
exactly the row space of that matrix. This check uses a separate elimination
routine and does not require the verifier to reproduce the particular basis
chosen when the certificate was generated.

The certificates may be regenerated by enumerating the reduced
row-echelon bases of all $(r-5)$-dimensional subspaces and accepting the first
one whose nonzero vectors avoid $A+A$. There are $97{,}155$ three-dimensional
subspaces of $\mathbb Z_2^8$. The validity of a returned certificate depends
only on the direct count of sums, not on the completeness or order of this
search. A failure to find a certificate causes an error and is never treated
as a proof that a set does not tile.

\subsection{Perfect codes of length 15}

The source file \texttt{perfect15}, supplied with version~3 of
\cite{OstergardPottonen2009}, contains $5{,}983$ representatives. Its
SHA-256 digest is
\begin{verbatim}
3bcb77324a7770c44f64b785d67d8da0f003113451575e06ca24852c7e5084f3
\end{verbatim}
Each record contains $2{,}048$ hexadecimal words of length sixteen.
The coordinate represented by $2^3$ is always zero. Deleting this coordinate
gives the fifteen coordinates used by the programs.

The program \texttt{perfect\_scan.cpp} checks all $32{,}768$ sums of a
codeword and an element of $\{0,e_1,\ldots,e_{15}\}$, for every source record.
They are distinct, so each record is a perfect code. The program then finds
every translation period of the code. Since the representatives contain
zero, every period is itself a codeword, so it is enough to test all
$2{,}048$ candidate translations.

An independent calculation verifies each translation kernel using the
integer Walsh transform of the code indicator. If $C$ is a code and $M$
is its translation kernel, then
\[
M=(\operatorname{supp}W_C)^\perp.
\]
Indeed, translation by $u$ multiplies $W_C(v)$ by $(-1)^{u\cdot v}$.
Fourier inversion shows that it preserves $C$ exactly when $u$ annihilates
the Fourier support. The program computes a basis for the span of this
support, checks orthogonality of every period to this basis, and checks the
dimensions. Thus the two calculations determine the same entire kernel.

The $69$ records whose kernels have dimension at least seven have distribution
\[
\begin{array}{c|rrrr|r}
\dim M&7&8&9&11&\text{total}\\ \hline
\text{number of representatives}&48&17&3&1&69.
\end{array}
\]
The full extraction agrees with every supplied selected record.
For each selected code, \texttt{perfect\_verify.py} recomputes its kernel,
chooses a basis of $M^\perp$, and uses this basis to realize the quotient
$\mathbb Z_2^{15}/M$. The sixteen images of $0,e_1,\ldots,e_{15}$ are
distinct. The supplied certificate gives $11-\dim M$ independent vectors
in this quotient, and the verifier checks all sums with the sixteen images.
Every point of the quotient occurs exactly once. All $69$ certificates pass.
The perfect-code reduction in the main text then lifts these complements
to the original sixteen-point tiles.

\subsection{Reproducing the checks}

The ancillary directory contains the selected source records, the positive
certificates, and the programs that generate and verify them. From that
directory, the command
\begin{verbatim}
python3 run_checks.py
\end{verbatim}
checks the supplied records and certificates. It requires Python~3.10 or
later and only the standard library. The command
\begin{verbatim}
python3 run_checks.py --full
\end{verbatim}
also retrieves the complete classification files, checks their digests,
and repeats both complete Hadamard scans and the complete perfect-code
scan with the independent Fourier calculation. This additionally requires
a C++17 compiler and the zlib development files. Already downloaded source
files may instead be specified on the command line, as described in the
ancillary README. The option \texttt{--regenerate} reconstructs the positive
subgroup certificates before verifying them.

The programs reject malformed records, repeated identifiers, missing
certificates, and incorrect sums. Their essential checks do not use Python
assertions and remain active with optimization enabled. The complete scans
and all positive certificate checks were run for this draft. Their results,
including the distinction between checking selected records and checking
complete-source coverage, are recorded with the ancillary files.

\end{document}